\documentclass{article}

\usepackage{arxiv}

\usepackage[utf8]{inputenc} 
\usepackage[T1]{fontenc}    
\usepackage{hyperref}       
\usepackage{url}            
\usepackage{booktabs}       
\usepackage{amsfonts}       
\usepackage{nicefrac}       
\usepackage{microtype}      
\usepackage{lipsum}
\usepackage{graphicx}
\graphicspath{ {./images/} }
\usepackage{amsmath}
\usepackage{amsthm}
\newtheorem{theorem}{Theorem}[section]
\newtheorem{proposition}[theorem]{Proposition}
\newtheorem{definition}[theorem]{Definition}

\numberwithin{theorem}{section}
\numberwithin{equation}{section}
\newtheorem{example}[theorem]{Example}

\title{Derivative for Functions $f : G \to H$, Where $G$ Is a Metric Divisible Group}

\author{
H\'ector Andr\'es Granada D\'iaz \\
Departamento de Matem\'aticas y Estad\'istica\\
Universidad del Tolima, Grupo de Matem\'aticas del Tolima (Grupo-MaT)\\
Barrio Santa Helena Parte Alta Cl 42 1-02\\
Ibagu\'e 730006299, Colombia\\
\texttt{ }\\
\And
Sime\'on Casanova Trujillo\\
Grupo de Investigaci\'on C\'alculo Cient\'ifico y Modelamiento Matem\'atico\\
Universidad Nacional de Colombia, Sede Manizales\\
Manizales 170003, Colombia\\
\texttt{scasanovat@unal.edu.co}\\
\And
Fredy E. Hoyos\\
Departamento de Energ\'ia El\'ectrica y Autom\'atica, Facultad de Minas\\
Universidad Nacional de Colombia, Sede Medell\'in\\
Carrera 80 No. 65-223, Robledo\\
Medell\'in 050041, Colombia\\
\texttt{ }
}
\begin{document}
\maketitle

\noindent\textbf{Citation:} Granada D\'iaz, H.A.; Trujillo, S.C.; Hoyos, F.E. Derivative for Functions $f : G \to H$, Where $G$ Is a Metric Divisible Group. \textit{Mathematics} 2022, 10, 3559. \url{https://doi.org/10.3390/math10193559}.\\[0.3em]
\noindent\textbf{Copyright:} \copyright{} 2022 by the authors. This article is an open access article distributed under the terms and conditions of the Creative Commons Attribution (CC BY) license (\url{https://creativecommons.org/licenses/by/4.0/}).\\

\begin{abstract}
In this paper, a derivative for functions $f : G \to H$, where $G$ is any metric divisible group and $H$ is a metric Abelian group with a group metric, is defined. Basic differentiation theorems are stated and demonstrated. In particular, we obtain the Chain Rule
\end{abstract}


\section{Introduction}

The concept of a derivative is one of the most fundamental in mathematics. Its definition is based on the concept of the limit of a function~\cite{botsko1985}. However, it is known that a formulation equivalent to the concept of a derivative can be obtained without having to use the concept of a limit. 

For example, in~\cite{caratheodory1954} an equivalent version of the derivative is presented, using only the topological concept of continuity and the algebraic concept of ``factorization'', which in his honor is known as the Carathéodory derivative.

In~\cite{kuhn1991}, Kuhn again takes up the definition given in~\cite{caratheodory1954} and shows the strength of this definition in the proof of the classical theorems of differential calculus. In~\cite{frolicher1966}, the concept of a derivative is presented for functions defined between pseudo-topological vector spaces through the use of Filters Theory, and the basic properties of the derivative are demonstrated.

The Carathéodory derivative has been generalized to functions $f : \mathbb{R}^n \to \mathbb{R}^m$~\cite{acosta1994}, where the equivalence between the Fréchet derivative and Carathéodory derivative is proved as well as the Inverse Function Theorem. The Carathéodory derivative has been defined for functions between Banach spaces~\cite{arora2021, cabrales2006}. 

For example, in~\cite{arora2021}, using this generalization, Schwarz's Lemma on the equality of mixed partial derivatives and Banach's Fixed Point Theorem are proved, and in~\cite{cabrales2006}, the derivative of Carathéodory allows the determination of the derivative of certain norms in function spaces.

In~\cite{pinzon1999}, the equivalence in $\mathbb{R}^2$ between Gateaux, Fréchet and Carathéodory derivatives is presented using a special topology in $\mathbb{R}^2$ that guarantees the continuity of the Gateaux derivative (radial topology). 

Since the derivative of Carathéodory involves the topological components of continuity and algebraic of factorization, in~\cite{acosta1996} the derivative of Carathéodory is generalized functions with a domain and co-domain as topological groups, showing among other things the important chain rule, and sufficient conditions are given to the uniqueness of the derivative.

Let us remember that a topological group or continuous group results from the union of two important mathematical structures: the topological space and the group. The union is in the sense that the group operation and the inversion operation must be continuous with the given topology~\cite{dikranjan2003, pontriaguin1986, spindler1994}.

If in the definition of topological group we consider the metric structure in exchange for the topological structure~\cite{lima1993}, we then obtain a metric group, so it is possible to also have a definition of the Carathéodory differentiability for functions between metric groups. 

In~\cite{granada2006}, the definition of differentiability given in~\cite{acosta1996} is used, and a derivative is defined in metric groups. However, in~\cite{granada2006}, it was not possible to prove the chain rule.

However, in~\cite{granada2007}, a result related to the convergence of certain homomorphisms is presented, which is the key to being able to demonstrate the chain rule, and it is what we will show in this work.

In this paper, we return to the definition of differentiability between metric groups given in~\cite{granada2006}, and we endow the $\tilde{\mathrm{Hom}}(G; H)$ space with an adequate metric, showing the chain rule and the linearity of the derivative. In particular, to show the homogeneity of the derivative of a function between metric groups, we must choose the co-domain of that function as a topological vector space.

\section{Group of Continuous Homomorphisms and Group Metric}

In this section, we will endow the space of continuous homomorphisms with a metric and a binary operation so that this space is a group. The above will allow us to construct a derivative in metric groups.

The space of continuous homomorphisms between metric groups is defined as:
\[
\tilde{\mathrm{Hom}}(G; H) = \{\phi : G \to H : \phi \text{ is a continuous homomorphism in } e \in G\}
\]
where $e$ represents the identity element of the metric group $G$.

As the application $\sigma : G \to H$ defined as $\sigma[x] = e_H$ belongs to $\tilde{\mathrm{Hom}}(G; H)$, it follows that $\tilde{\mathrm{Hom}}(G; H) \neq \emptyset$.

The next definition is a generalization of the one presented in~\cite{lima1993} (p. 159), and we will use it to prove Propositions~\ref{prop1} and~\ref{prop5}.

\begin{definition}
Let $(H, \ast, d)$ be a metric group and $e \in H$ the identity element. We define by group metric the metric of $H$, if for $k$ fixed in $H$ and for all $x, y \in H$ we have:
\begin{enumerate}
\item[(1)]
$d(x \ast y, e) \leq d(x, e)d(y, e) + d(x, e) + d(y, e)$.
\item[(2)]
There exists $c_k > 0$ such that $d(x \ast k, y \ast k) \leq c_k d(x, y)$.
\end{enumerate}
\end{definition}

As examples of group metrics, we have, among others, the metric induced by the usual norm in $(\mathbb{C}^*, \cdot)$, and the metric induced by the usual norm in $(\mathbb{R}, +)$. 

In $(M_{n\times n}(\mathbb{R}), +)$, we define the usual metric as follows:
\[
d(X, Y) = \|X - Y\|
\]
Here $\|\cdot\|$ is the Euclidean norm. Therefore, $d$ is a metric group. To see this, one can observe that:
\[
d(X + Y, 0) = \|X + Y\| \leq \|X\| + \|Y\| \leq \|X\| \cdot \|Y\| + \|X\| + \|Y\| = d(X, 0)d(Y, 0) + d(X, 0) + d(Y, 0)
\]
On the other hand, if $K$ is fixed in $(M_{n\times n}(\mathbb{R}), +)$, we have that
\[
d(X + K, Y + K) = \|(X + K) - (Y + K)\| = \|X - Y\| = d(X, Y)
\]
Therefore, condition (2) in Definition~\ref{def1} is satisfied if we let $c_K = 1$.

In what follows, we consider $H$ as an Abelian metric group with a group metric.

\begin{definition}
In the set $\tilde{\mathrm{Hom}}(G, H)$, we define the following metric:
\[
\tilde{d}(\phi, \psi) = \sup_{t\in G} \frac{d_H(\phi[t], \psi[t])}{1 + d_H(\phi[t], \psi[t])}
\]
\end{definition}

Next we will endow $\tilde{\mathrm{Hom}}(G; H)$ with a binary operation, and later we will show that this space with this binary operation is an Abelian group.

\begin{definition}
We define in $\tilde{\mathrm{Hom}}(G; H)$ the following binary operation:
\[
\oplus: \tilde{\mathrm{Hom}}(G; H) \times \tilde{\mathrm{Hom}}(G; H) \to \tilde{\mathrm{Hom}}(G; H),\quad (\phi, \psi) \mapsto \phi \oplus \psi
\]
where $\phi \oplus \psi : G \to H$ is defined as $(\phi \oplus \psi)[x] = \phi[x] \ast \psi[x]$ and $\ast$ is the binary operation on $H$.
\end{definition}

\begin{proposition}\label{prop1}
$\oplus$ is well defined.
\end{proposition}
\begin{proof}
Let $(\phi, \psi), (\lambda, \theta) \in \tilde{\mathrm{Hom}}(G; H) \times \tilde{\mathrm{Hom}}(G; H)$ such that $(\phi, \psi) = (\lambda, \theta)$. Therefore, $\phi = \lambda$ y $\psi = \theta$. For $x \in G$, we have that $(\phi \oplus \psi)[x] = \phi[x] \ast \psi[x] = \lambda[x] \ast \theta[x] = (\lambda \oplus \theta)[x]$, i.e., $\phi \oplus \psi = \lambda \oplus \theta$, or $\oplus(\phi, \psi) = \oplus(\lambda, \theta)$. 

If $x = y \in G$, then of course $(\phi \oplus \psi)[x] = \phi[x] \ast \psi[x]$. As $\phi$ and $\psi$ are functions, then $\phi[x] = \phi[y]$ and $\psi[x] = \psi[y]$, and therefore, $(\phi \oplus \psi)[x] = (\phi \oplus \psi)[y]$.

Next, we claim that $\phi \oplus \psi \in \tilde{\mathrm{Hom}}(G; H)$. Let $x, y \in G$ be. Then:
\begin{enumerate}
\item[(i)]
\[
(\phi \oplus \psi)[x \cdot y] = \phi[x \cdot y] \ast \psi[x \cdot y] = \{\phi[x] \ast \phi[y]\} \ast \{\psi[x] \ast \psi[y]\} = \{\phi[x] \ast \psi[x]\} \ast \{\phi[y] \ast \psi[y]\} = (\phi \oplus \psi)[x] \ast (\phi \
\oplus \psi)[y].
\]
\item[(ii)]
As $\phi, \psi \in \tilde{\mathrm{Hom}}(G; H)$, $\phi$ and $\psi$ are continuous. Let $\epsilon > 0$ be and $\varepsilon = \min\{\epsilon, 1\}$. For $\varepsilon$ there are $\delta_1, \delta_2 > 0$ such that:
\[
d_G(x, e) < \delta_1 \implies d_H(\phi[x], e_H) < \varepsilon/2 \quad \text{and} \quad d_G(x, e) < \delta_2 \implies d_H(\psi[x], e_H) < \varepsilon/3.
\]
If we do $\delta = \min\{\delta_1, \delta_2\}$, then $d_G(x, e) < \delta$ implies:
\begin{align*}
d_H\bigl((\phi \oplus \psi)[x], e_H\bigr)
&= d_H\bigl(\phi[x] \ast \psi[x], e_H\bigr) \\
&\le d_H(\phi[x], e_H)\, d_H(\psi[x], e_H)
   + d_H(\phi[x], e_H)
   + d_H(\psi[x], e_H) \\
&< \frac{\varepsilon^2}{6}
   + \frac{\varepsilon}{2}
   + \frac{\varepsilon}{3} \\
&= \varepsilon^2 + \frac{5\varepsilon}{6}
 < \varepsilon \le \varepsilon.
\end{align*}

\end{enumerate}
\end{proof}

The next proposition will be used to prove the existence of invertible elements in $\tilde{\mathrm{Hom}}(G; H)$, and its proof is similar to the one presented in~\cite{arora2021}.

\begin{proposition}\label{prop2}
Let $\phi \in \tilde{\mathrm{Hom}}(G; H)$ and $h : G \to G$ be a continuous function in a $\in G$. Then
\[
\lim_{x\to a} \phi[h(x)] = \phi[h(a)].
\]
\end{proposition}

The proof follows from the continuity of $\phi$ at $h(a)$ and continuity of $h$ at $a$.

\begin{proposition}\label{prop3}
$(\tilde{\mathrm{Hom}}(G; H), \oplus)$ is Abelian group.
\end{proposition}

In effect, the application $\sigma : G \to H$ defined as $\sigma[x] = e_H$ for all $x \in G$, is the module of $(\tilde{\mathrm{Hom}}(G; H), \oplus)$. The function $\phi^{-1} : G \to H$ defined by $\phi^{-1}[x] = \phi[x^{-1}]$ is the inverse of $\phi$. 

We also have that $\phi^{-1}$ is a homomorphism of $G$ in $H$. The associativity and commutativity in $H$ allow us to see that the operation $\oplus$ is associative and commutative.

Furthermore, since $G$ is a metric group then the $\mathrm{inv} : G \to G$ defined by $\mathrm{inv}(x) = x^{-1}$ function is continuous, and by the Proposition~\ref{prop2} statement we have:
\[
\lim_{x\to e} \phi^{-1}[x] = \lim_{x\to e} \phi[x^{-1}] = \lim_{x\to e} \phi[\mathrm{inv}(x)] = \lim_{x\to e} \phi[\mathrm{inv}(e)] = \phi[e^{-1}] = \phi^{-1}[e]
\]
We have thus shown that $\phi^{-1} \in \tilde{\mathrm{Hom}}(G; H)$.

We use the following theorem to prove the linearity of the derivative.

\begin{theorem}\label{thm1}
Let $G$ and $H$ be metric groups, $G$ locally compact, $\phi : G \to \tilde{\mathrm{Hom}}(G; H)$. Let $t \in G$ be fixed. Then
\[
\lim_{x\to a} \phi(x)[t] = \phi(a)[t] \implies \lim_{x\to a} \phi(x) = \phi(a)
\]
that is, $\phi$ is continuous at $a$.
\end{theorem}

\section{Construction of a Derivative in Metric Groups}

\begin{definition}[Definition 4]\label{def4}
Let $G$ and $H$ be metric groups. A function $f : G \to H$ is differentiable at $a \in G$ if there is a neighborhood $V_a$ of $a$ and a function $\phi_f : G \to \tilde{\mathrm{Hom}}(G; H)$ continuous at $a$ such that:
\[
f(x) f(a)^{-1} = \phi_f(x)[x a^{-1}] \quad \text{for all } x \in V_a.
\]
The derivative of $f$ at $a$ is $\phi_f(a)$, and $\phi_f$ is a slope function for $f$ at $a$. The latter is Carathéodory's characterization of differentiability between metric groups.
\end{definition}

In this definition, $\tilde{\mathrm{Hom}}(G; H)$ is endowed with the metric given in Definition~\ref{def2}.

To guarantee the uniqueness of the derivative of a function $f : G \to H$, we use the following definition.

\begin{definition}[Definition 5]\label{def5}
A group $G$ is divisible if for every element $g \in G$ and every $n \in \mathbb{N}$ there exists $x \in G$ such that the equation $g = x^n$ has only one solution.

Furthermore, we impose the condition $\lim_{n\to\infty} x^{1/n} = e$ for any $x \in G$.

If the operation of the group $G$ is given in additive form, we will write $g = n x$ and $\lim_{n\to\infty} x/n = e$.
\end{definition}

For example, $(\mathbb{R}, +)$, $(S^1, \cdot)$, $(M_{n\times n}(\mathbb{R}), +)$ are divisible groups.

The next proposition will help us to show the uniqueness of the derivative.

\begin{proposition}\label{prop4}
Let $G$ be a divisible metric group and $\phi : G \to \tilde{\mathrm{Hom}}(G; H)$ continuous on $a \in G$. Then $\lim_{n\to\infty} \phi(z^{1/n} a) = \phi(a)$.
\end{proposition}

Indeed, since the $\lambda_a : G \to G$ defined by $\lambda_a(x) = x a$ is continuous at $G$ (because it is homeomorphism) and $\phi$ is continuous at $a$, the result follows as $G$ is divisible.

Regarding the uniqueness of the derivative, we have the following result, proof of which can be found in~\cite{granada2006}.

\begin{theorem}[Uniqueness of the derivative]\label{thm2}
Let $G$ be a divisible metric group, $H$ Abelian metric group with a group metric. If $f : G \to H$ is differentiable at $a \in G$, then $\phi_f(a)$ is unique.
\end{theorem}

Proposition~\ref{prop5} and Theorem~\ref{thm3} will allow us to show that differentiability implies continuity.

\begin{proposition}\label{prop5}
Let $\phi : G \to H$ be continuous at $a \in G$ and $k \in H$. Then $\lim_{x\to a} k \ast \phi(x) = k \ast \phi(a)$.
\end{proposition}
\begin{proof}
Since $H$ has a group metric, for $x, y \in H$ and $k \in H$ there is $c_k > 0$ such that $d_H(k \ast x, k \ast y) \leq c_k d_H(x, y)$. Because $\phi$ is continuous at $a$, for $\varepsilon > 0$ there is $\delta > 0$ such that:
\[
d_G(x, a) < \delta \implies d_H(\phi(x), \phi(a)) < \frac{\varepsilon}{c_k}.
\]
So, if $d_G(x, a) < \delta$ then $d_H(k \ast \phi(x), k \ast \phi(a)) \leq c_k d_H(\phi(x), \phi(a)) < \varepsilon$, i.e.,
\[
\lim_{x\to a} k \ast \phi(x) = k \ast \phi(a).
\]
\end{proof}

The proof of Theorem~\ref{thm3} is similar to those in~\cite{arora2021} and will be omitted.

\begin{theorem}\label{thm3}
Let $\phi : G \to \tilde{\mathrm{Hom}}(G; H)$ be a continuous function at $a \in G$, $h : G \to G$ continuous at $a$. Then $\lim_{x\to a} \phi(x)[h(x)] = \phi(a)[h(a)]$.
\end{theorem}

We will use Theorem~\ref{thm4} to prove the additivity of the derivative.

\begin{theorem}\label{thm4}
Let $\phi, \psi : G \to \tilde{\mathrm{Hom}}(G; H)$ be continuous in $a \in G$, $d_{H\times H}$, the product metric defined by $d_{H\times H}((x, y), (a, b)) = d_H(x, a) + d_H(y, b)$. If $w(x) = \phi(x) \oplus \psi(x)$ and $m \in G$, then $\lim_{x\to a} w(x)[m] = w(a)[m]$.
\end{theorem}
\begin{proof}[Proof of Theorem~\ref{thm4}]
Since $H$ is a metric space, we have that
\[
\lim_{(r,t)\to(c,d)} r \ast t = c \ast d.
\]
Let $m \in G$ be fixed. If we perform in the previous limit $r = \phi(x)[m]$, $t = \psi(x)[m]$, $c = \phi(a)[m]$, $d = \psi(a)[m]$, then by the definition of limit, for $\epsilon > 0$ there exists $\delta_1 > 0$ such that:
\[
d_{H\times H}((\phi(x)[m], \psi(x)[m]), (\phi(a)[m], \psi(a)[m])) = d_H(\phi(x)[m], \phi(a)[m]) + d_H(\psi(x)[m], \psi(a)[m]) < \delta_1
\]
which implies that $d_H(\phi(x)[m] \ast \psi(a)[m], \phi(a)[m] \ast \psi(a)[m]) < \epsilon$.

Since $\phi, \psi$ are continuous at $a$, then for $\delta_1/2 + \delta_1$ there exists $\delta_2, \delta_3 > 0$ such that:
\[
d_G(x, a) < \delta_2 \implies \tilde{d}(\psi(x), \psi(a)) < \frac{\delta_1}{2 + \delta_1}, \quad d_G(x, a) < \delta_3 \implies \tilde{d}(\phi(x), \phi(a)) < \frac{\delta_1}{2 + \delta_1}.
\]
If we let $\delta = \min\{\delta_2, \delta_3\}$, then $d_G(x, a) < \delta$ implies that:
\[
\frac{d_H(\phi(x)[m], \phi(a)[m])}{1 + d_H(\phi(x)[m], \phi(a)[m])} < \tilde{d}(\phi(x), \phi(a)) < \frac{\delta_1}{2 + \delta_1},
\]
from which $d_H(\phi(x)[m], \phi(a)[m]) < \delta_1/2$. Similarly, $d_G(x, a) < \delta$ implies $d_H(\psi(x)[m], \psi(a)[m]) < \delta_1/2$.

Then, $d_G(x, a) < \delta$ implies:
\[
d_H(\phi(x)[m], \phi(a)[m]) + d_H(\psi(x)[m], \psi(a)[m]) = d_{H\times H}((\phi(x)[m], \psi(x)[m]), (\phi(a)[m], \psi(a)[m])) < \delta_1.
\]
Therefore, for $\epsilon > 0$ there exists $\delta > 0$ such that $d_G(x, a) < \delta$ implies
\[
d_H(\phi(x)[m] \ast \psi(x)[m], \phi(a)[m] \ast \psi(a)[m]) = d_H(w(x)[m], w(a)[m]) < \epsilon
\]
that is,
\[
\lim_{x\to a} w(x)[m] = w(a)[m].
\]
\end{proof}

We use Theorem~\ref{thm5} to demonstrate the important chain rule. This proof is analogous to Theorem~\ref{thm4} and is omitted.

\begin{theorem}\label{thm5}
Let $M$ be an Abelian metric group with a group metric, $f : G \to H$ continuous at $a \in G$, $g : H \to M$ continuous at $f(a) \in H$, $\phi_g : H \to \tilde{\mathrm{Hom}}(H; M)$ continuous at $f(a)$ and $\phi_f : G \to \tilde{\mathrm{Hom}}(G; H)$ continuous at $a$. Therefore, for all $m \in G$, we have that:
\[
\lim_{x\to a} \phi_g(f(x))[\phi_f(x)[m]] = \phi_g(f(a))[\phi_f(a)[m]].
\]
\end{theorem}
\section{Basic Theorems of Differentiation}

In this section, we will demonstrate the basic theorems of differentiation such as: differentiability implies continuity, the linearity of the derivative and the chain rule. To achieve this, we naturally define, for $f, g : G \to H$ and $k \in H$, $(f + g)(x) = f(x) \ast g(x)$, where $\ast$ represents the binary operation in $H$. Although the following theorems appear in~\cite{arora2021}, for convenience they are presented here.

\begin{theorem}[Differentiability implies continuity]\label{thm6}
Let $f : G \to H$ be differentiable at $a \in G$. Then $f$ is continuous at $a$.
\end{theorem}
\begin{proof}
Since $f$ is differentiable at $a$, there exists $\phi_f : G \to \tilde{\mathrm{Hom}}(G; H)$ continuous at $a$ such that:
\[
f(x) \ast f(a)^{-1} = \phi_f(x)[x a^{-1}]
\]
for $x$ in some neighborhood of $a$. Then, in this neighborhood, we have:
\[
f(x) = \phi_f(x)[x a^{-1}] \ast f(a) = f(a) \ast \phi_f(x)[x a^{-1}] = f(a) \ast \phi_f(x)[\lambda_{a^{-1}}(x)].
\]
If we do $x \to a$ and use Proposition~\ref{prop5}, we have that:
\[
\lim_{x\to a} f(x) = f(a) \ast \lim_{x\to a} \phi_f(x)[\lambda_{a^{-1}}(x)].
\]
Next, by Theorem~\ref{thm3} we have that:
\[
\lim_{x\to a} f(x) = f(a) \ast \phi_f(a)[\lambda_{a^{-1}}(a)] = f(a) \ast \phi_f(a)[a \cdot a^{-1}] = f(a) \ast \phi_f(a)[e] = f(a) \ast e_H = f(a).
\]
\end{proof}

\begin{theorem}[Linearity of the derivative]\label{thm7}
Assume that $f, g : G \to H$ are differentiable at $a \in G$. Then $f + g$ is differentiable at $a$ and $(f + g)'(a) = f'(a) \oplus g'(a)$.
\end{theorem}
\begin{proof}
Since $f, g$ are differentiable at $a$, there are $V_a$ and $B_a$ neighborhoods of $a$ and $\phi_f, \phi_g : G \to \tilde{\mathrm{Hom}}(G; H)$ continuous at $a$ such that:
\[
f(x) \ast f(a)^{-1} = \phi_f(x)[x a^{-1}] \quad \text{for all } x \in V_a
\]
\[
g(x) \ast g(a)^{-1} = \phi_g(x)[x a^{-1}] \quad \text{for all } x \in B_a
\]
Now, for $x \in V_a \cap B_a$ we have:
\begin{align*}
(f + g)(x) \ast [(f + g)(a)]^{-1} &= (f(x) \ast g(x)) \ast (f(a) \ast g(a))^{-1}\\
&= f(x) \ast g(x) \ast f(a)^{-1} \ast g(a)^{-1}\\
&= (f(x) \ast f(a)^{-1}) \ast (g(x) \ast g(a)^{-1})\\
&= \phi_f(x)[x a^{-1}] \ast \phi_g(x)[x a^{-1}]\\
&= (\phi_f(x) \oplus \phi_g(x))[x a^{-1}]
\end{align*}
As $(\tilde{\mathrm{Hom}}(G; H), \oplus)$ is a group, $\phi_f(x) \oplus \phi_g(x) \in \tilde{\mathrm{Hom}}(G; H)$. Let $\phi_{(f+g)}(x) = \phi_f(x) \oplus \phi_g(x)$. By Theorem~\ref{thm4} we have that $\phi_{(f+g)} : G \to \tilde{\mathrm{Hom}}(G; H)$ is continuous in $a$ and so much
\[
(f + g)'(a) = \phi_{(f+g)}(a) = \phi_f(a) \oplus \phi_g(a) = f'(a) \oplus g'(a).
\]
\end{proof}

If we want to obtain a version of the homogeneity of the derivative, we must provide the metric group with a multiplication by scalar compatible with the operation of the group. In this way, the metric group is nothing more than a topological vector space with the topology induced by the metric.

Suppose, then that in the metric group $(H, \hat{\ast})$ a vector space structure $K$ and a continuous application $(\lambda, x) \mapsto \lambda x$ of $K \times H$ in continuous $H$, where in space $K \times H$ the product metric is considered.

Note that the vector space structure automatically requires the group $(H, \hat{\ast})$ to be Abelian.

If $H$ is a topological vector space over the field $K$ and $f : G \to H$ is a function, we define $\alpha f : G \to H$ as $(\alpha f)(x) = \alpha f(x)$ for $\alpha \in K$.

We thus have the following result:

\begin{theorem}[Homogeneity of the derivative]\label{thm8}
Let $(G, \ast)$ be a metric group and $(H, \hat{\ast})$ topological vector space over field $K$. If $f$ is differentiable in $a \in G$, then $\alpha f$ is differentiable at $a$ and also $(\alpha f)'(a) = \alpha f'(a)$
\end{theorem}
\begin{proof}
\[
(\alpha f)(x) \hat{\ast} ((\alpha f)(a))^{-1} = \alpha f(x) \hat{\ast} \alpha f(a)^{-1} = \alpha \bigl( f(x) \hat{\ast} f(a)^{-1} \bigr) = \alpha \psi_f(x)[x \ast a^{-1}]
\]
It only remains to be seen that the application $\alpha \psi_f : G \to \tilde{\mathrm{Hom}}(G; H)$ defined as $\bigl( \alpha \psi_f \bigr)(x) = \alpha \psi_f(x)$ is continuous at $a$, where $\psi_f : G \to \tilde{\mathrm{Hom}}(G; H)$ is continuous at $a$.

Indeed, given the continuity of $\psi_f$ in $a$, we have $\psi_f(x) \to \psi_f(a)$ when $x \to a$, and hence, for $y \in G$, we have that $\psi_f(x)[y] \to \psi_f(a)[y]$.

Let us now set $\alpha \in K$ and $x \to a$, then we have that $\bigl( \alpha, \psi_f(x)[y] \bigr) \to \bigl( \alpha, \psi_f(a)[y] \bigr)$, from where $\alpha \psi_f(x)[y] \to \alpha \psi_f(a)[y]$ for all $y \in G$. Now, by Theorem~\ref{thm1} it follows that $\alpha \psi_f(x) \to \alpha \psi_f(a)$.
\end{proof}

\begin{theorem}[Chain rule]\label{thm9}
Let $M$ be an Abelian metric group, $f : G \to H$ differentiable at $a \in G$, $g : H \to M$ differentiable at $f(a) \in H$. Then $g \circ f : G \to M$ is differentiable at $a$ and $(g \circ f)'(a) = g'(f(a)) \circ f'(a)$.
\end{theorem}
\begin{proof}
Since $f$ is differentiable at $a$ and $g$ is differentiable at $f(a)$, there are $\phi_f : G \to \tilde{\mathrm{Hom}}(G; H)$ continuous at $a$ and $\phi_g : H \to \tilde{\mathrm{Hom}}(H; M)$ continuous at $f(a)$ such that
\[
f(x) \ast f(a)^{-1} = \phi_f(x)[x a^{-1}]
\]
and
\[
g(f(x)) + g(f(a))^{-1} = \phi_g(f(x))[f(x) \ast f(a)^{-1}],
\]
where $+$ is the binary operation of $M$ (usual notation due to $M$ is an Abelian group).

On the other hand,
\begin{align*}
(g \circ f)(x) + [(g \circ f)(a)]^{-1} &= g(f(x)) + g(f(a))^{-1}\\
&= \phi_g(f(x))[f(x) \ast f(a)^{-1}]\\
&= \phi_g(f(x))[\phi_f(x)[x a^{-1}]]\\
&= (\phi_g(f(x)) \circ \phi_f(x))[x a^{-1}].
\end{align*}
Let $\phi_{(g\circ f)}(x) = \phi_g(f(x)) \circ \phi_f(x) \in \tilde{\mathrm{Hom}}(G; M)$. By Theorems~\ref{thm1} and~\ref{thm5}, we have that $\phi_{(g\circ f)} : G \to \tilde{\mathrm{Hom}}(G; H)$ is continuous on $a$ and so
\[
(g \circ f)'(a) = \phi_{(g\circ f)}(a) = \phi_g(f(a)) \circ \phi_f(a) = g'(f(a)) \circ f'(a).
\]
\end{proof}

Notice the impressive proof of the chain rule.
In order to illustrate the exposed theory, we consider two examples, as follows:

\begin{example}
Let $G = M_{n\times n}(\mathbb{R})$ with a usual matrix addition, which is a divisible and Abelian metric group with a usual metric, defined as above. In the first place, let us consider the function
\[
f : G \to G, \quad X \mapsto X^2
\]
and let us see that it is differentiable at every point $A$ of $G$. Indeed, let
\[
\phi_f(X) : G \to G, \quad Y \mapsto \phi_f(X)(Y) = A Y + Y X
\]
then we have that
\[
\phi_f(X)[X - A] = A(X - A) + (X - A)X = X^2 - A^2 = f(X) - f(A)
\]
$\phi_f(X)$ is a homomorphism. Indeed,
\[
\phi_f(X)[Y + W] = A(Y + W) + (Y + W)X = (A Y + Y X) + (A W + W X) = \phi_f(X)[Y] + \phi_f(X)[W]
\]
It only remains to show that $\phi_f$ is continuous at $A$. By Theorem~\ref{thm1}, it is enough to show that
\[
\lim_{X\to A} \phi_f(X)[Y] = \phi_f(A)[Y]
\]
we have that
\[
d\bigl( \phi_f(X)[Y], \phi_f(A)[Y] \bigr) = \|\phi_f(X)[Y] - \phi_f(A)[Y]\| = \|(A Y + Y X) - (A Y + Y A)\| \leq \|Y\| \cdot \|X - A\|
\]
Therefore, for $\epsilon > 0$ and $Y \in G$, $Y \neq 0$, if we perform $\delta = \epsilon/\|Y\|$ we have that $\|\phi_f(X)[Y] - \phi_f(A)[Y]\| < \epsilon$. We have shown that $f$ is differentiable at $A$ and $f'(A) = \phi_f(A)$, where $\phi_f(A)[Y] = A Y + Y A$. For $Y = 0$, the result is clear.

In this example, the derivative is equal to the Fréchet derivative.
\end{example}

\begin{example}
Let $G$ be an Abelian topological group and
\[
f : G \to G, \quad x \mapsto x^3
\]
we have that:
\begin{align*}
f(x) f(a)^{-1} &= x^3 a^{-3}\\
&= x(x a^{-1} x a^{-1} a^{-1}) a^{-1}\\
&= (x a^{-1}) a (x a^{-1} x a^{-1} a^{-1}) a^{-1}\\
&= (x a^{-1}) a d_a (x a^{-1} x a^{-1} a^{-1})\\
&= \bigl( i_G \cdot a d_a (i_G \cdot a d_a) \bigr) [x a^{-1}]
\end{align*}
Here, $i_G : G \to G$ is the identity function, $a d_a : G \to G$ is defined by $a d_a(x) = a x a^{-1}$ and $i_G \cdot a d_a : G \to G$ is defined by $(i_G \cdot a d_a)(x) = i_G(x) \cdot a d_a(x)$. It is clear that both $i_G$ and $a d_a$ are continuous. Therefore, function $\phi(x) = i_G \cdot a d_a (i_G \cdot a d_a)$ is continuous. Furthermore, because $G$ is an Abelian group, $\phi(x)$ is a homomorphism. In summary, we obtain that $\phi(x) \in \tilde{\mathrm{Hom}}(G, H)$.

Finally, if $\phi_f : G \to \tilde{\mathrm{Hom}}(G; H)$ defined by $\phi_f(x) = \phi(x)$ were continuous at $a$ (this depends on the metric chosen in $G$), then function $f$ is differentiable at $a$ in the sense of Definition~\ref{def4}, with $f'(a) = \phi(a)$.
\end{example}

\section{Conclusions}

Carathéodory differentiability for functions $f : \mathbb{R} \to \mathbb{R}$ has the advantage of being able to be generalized in the context of metric groups. In this sense, a definition of differentiability in metric groups has been given, showing its most important properties, such as linearity and the chain rule. The theory is illustrated by calculating the derivative of a specific function between metric groups.

\noindent\textbf{Author Contributions:} Conceptualization, investigation, methodology, and software, H.A.G.D. and S.C.T.; Formal analysis, writing—review, and editing, H.A.G.D., S.C.T. and F.E.H. All authors have read and agreed to the published version of the manuscript.

\vspace{0.5em}
\noindent\textbf{Funding:} This research received no external funding.

\vspace{0.5em}
\noindent\textbf{Institutional Review Board Statement:} Not applicable.

\vspace{0.5em}
\noindent\textbf{Informed Consent Statement:} Not applicable.

\vspace{0.5em}
\noindent\textbf{Data Availability Statement:} Not applicable.

\vspace{0.5em}
\noindent\textbf{Acknowledgments:} The work of Héctor Andrés Granada Díaz was supported by Universidad del Tolima, Ibagué-Colombia. The work of Simeón Casanova Trujillo was supported by Universidad Nacional de Colombia, Sede Manizales. The work of Fredy E. Hoyos was supported by Universidad Nacional de Colombia, Sede Medellín.

\vspace{0.5em}
\noindent\textbf{Conflicts of Interest:} The authors declare no conflict of interest.

\vspace{2em}

\bibliographystyle{unsrt}  


\end{document}